\documentclass[11pt, a4paper]{amsart}
\usepackage{amsmath,amssymb,amsthm,amscd}
\usepackage[matrix,arrow, curve]{xy}
\numberwithin{equation}{section}
\usepackage{color}
\usepackage{graphicx}
\usepackage{tikz}
\usepackage[hypertexnames=false,linkcolor=blue, colorlinks=true, citecolor=blue, filecolor=blue,urlcolor=blue ,pagebackref]{hyperref}

\definecolor{softgray}{gray}{0.75}

\newcounter{CountAlpha}

\theoremstyle{plain}
\newtheorem{Thm}{Theorem}[section]

\newtheorem{Prop}[Thm]{Proposition}

\theoremstyle{definition}
\newtheorem{Def}[Thm]{Definition}

\newtheorem{Ex}[Thm]{Example}

\theoremstyle{remark}

\newcommand{\IQ}{\mathbf{Q}}

\newcommand{\K}{\mathbf{K}}

\newcommand{\cI}{\mathcal{I}}

\newcommand{\cO}{\mathcal{O}}

\newcommand{\cT}{\mathcal{T}}	
\newcommand{\cX}{\mathcal{X}}

\renewcommand{\u}{{\boldsymbol u}}

\newcommand{\x}{{\boldsymbol x}}

\newcommand{\Spec}{\operatorname{Spec}}

\definecolor{gruen}{rgb}{0, 0.666, 0}

\begin{document}

	\thispagestyle{empty}

	\title{Teissier singularities}

	\author{Hussein Mourtada}
	\thanks{The first named author is partially supported by the ANR-SINTROP
		}
	\address{
		Hussein Mourtada\\
		Universit\'e Paris Cité and Sorbonne Universit\'e, CNRS, IMJ-PRG, F-75013 Paris, France.}
	\email{hussein.mourtada@imj-prg.fr}

	\author{Bernd Schober}
	\address{Bernd Schober\\
	None (Hamburg, Germany)}
	\email{schober.math@gmail.com}

	\keywords{singularities, resolution of singularities, positive characteristic, characteristic polyhedron, overweight deformations.}
	\subjclass[2020]{14B05, 32S05, 14E15}

	\begin{abstract}
	The goal of this note is to introduce Teissier singularities and to explain why they are candidate to play, in positive characteristics, a role for resolution of singularities  which is similar to the role played by quasi-ordinary singularities in characteristic zero. 
	\end{abstract}

	\maketitle

	\section{Introduction}
	\label{Intro}

\noindent Let $\K$ be an algebraically closed field, that we assume for the moment to be of characteristic $0.$ Let $f\in \K[[\x]][z]$ (where $\K[[\x]]=\K[[x_1,\ldots,x_d]])$ be a Weierstrass polynomial, \textit{i.e.},
\[f(\x,z)=z^n+a_1(\x)z^{n-1}+\cdots + a_{n-1}(\x)z+a_{n}(\x),\]
where $a_i(\x) \in \K[[\x]];$ we assume $f(0)=0.$ The polynomial $f$ is said to be quasi-ordinary if its discriminant with respect to $z$ is a unit times a monomial in $\textbf{K}[[x_1,\ldots,x_d]]$ (we refer to this condition as the discriminant condition). The map 
$\pi:\{f=0\}\longrightarrow\textbf{K}^d$, which is the restriction on $\{f=0\}$ of the projection  $\K^{d+1}\longrightarrow\textbf{K}^d$ on the first $d$ coordinates, is finite and  its ramification locus is the zero locus of the discriminant of $f$ with respect to $z.$ When $f$ is quasi-ordinary, we say that $\{f=0\}$ has a quasi-ordinary singularity (with respect to the projection on the first $d$ variables) at the origin $O.$ Quasi-ordinary singularities appear in  Jung's method of resolution of singularities in characteristic $0:$ for any $f\in \textbf{K}[[\x]][z],$ this method is recursive on the dimension and consists in using embedded resolution of singularities in dimension $d-1,$ to transform the discriminant of $f$  into a normal crossing divisor (locally a unit times a monomial). The pull back of  $\{f=0\}$ by the resolution morphism of the discriminant  will then have only quasi-ordinary singularities and the resolution problem is  reduced to the problem of resolution of quasi-ordinary singularities and then to the problem of patching these local resolutions (see \cite{PP} for a survey on Jung's approach).\\

In \cite{MS}, we gave a characterization of quasi-ordinary singularities in terms of an invariant $\kappa(f)$ of $f.$  The construction of $\kappa(f)$ uses Hironaka's characteristic polyhedron and a weighted generalization of this polyhedron applied to successive embeddings of the singularity defined by $f$ in affine spaces of higher dimensions. Note that Hironaka's characteristic polyhedron is a projection of the classical Newton polyhedron (a convex body cooked up from the exponents of the monomials appearing in the expression of $f$), but it has some intrinsic properties thanks to the minimizing process explained in \cite{MS}. The invariant $\kappa(f)$ is a string whose components are all in $\textbf{Q}_{\geq 0}^d,$ except of the last one which is either $-1$ or $\infty.$ The size of $\kappa(f)$ depends on $f.$ In \cite{MS} we prove the following theorem:\\

\begin{Thm}Let $f$ be as above. The singularity $\{f=0\}$ is quasi-ordinary with respect to the projection $\{f=0\}\longrightarrow\textbf{K}^d$ if and only if the last component of  $\kappa(f)$ is $\infty.$ 
\end{Thm}

It is worth mentioning that when $\textbf{K}=\textbf{C},$ we prove that the invariant $\kappa(f)$ is a complete invariant of the topological type of $(\{f=0\},0)\subset (\textbf{C}^{d+1},0).$\\  
On the one hand, while we know how to resolve quasi-ordinary singularities (in characteristic 0), in positive characteristics, the singularities which satisfy the condition on the discriminant can be extremely wicked; e.g., hundreds of pages  of the proof of Cossart-Piltant of resolution in dimension $3$ are dedicated to this type of singularities (see \cite{CParithm} for the arithmetical case). So, in positive characteristics, the reduction of the resolution of singularities problem to the singularities satisfying the discriminant condition cannot be compared with Jung's approach in characteristic $0.$ On the other hand, while in characteristic $0,$ the last component of $\kappa(f)$ being $\infty$ is equivalent to $f$ being quasi-ordinary, in  characteristic $p,$ the invariant $\kappa(f)$ is still meaningful but  the condition on its last component gives rise to a different condition than the one given by the discriminant. 
This motivates us to define the following class of singularities \cite{MS1}:\\

\begin{Def}Let  $\textbf{K}$ be an algebraically closed field of characteristic $p>0.$ Let $f\in \textbf{K}[[x_1,\ldots,x_d]][z]$ be a Weierstrass polynomial satisfying $f(0)=0.$ The hypersurface singularity $(X,0)$ defined by $\{f=0\}$ is a Teissier singularity if the last component of $\kappa(f)$ is $\infty.$
\end{Def}

\noindent We named these singularities "Teissier singularities" because Teissier proved that along an Abhyankar rational valuation, any hypersurface singularity can be embedded in a higher dimensional affine space with a special type of equations that define an "overweight deformation" whose generic fiber is isomorphic to the singularity and whose special fiber is the toric variety associated with the graded algebra of the valuation \cite{T}; this also related to Teissier's viewpoint on resolution of singularities of plane curve singularities after re-embedding them in a higher dimensional affine space \cite{TC}. We expect that for a Teissier singularity defined by a polynomial $f,$
all the valuations which extend rational monomial valuations on $\textbf{K}[[\x]]$ to $\textbf{K}[[x]][z]/(f)$ induce the  "same overweight deformation" and this property characterizes them \cite{CMT, CMT1}. Teissier singularities do not satisfy the discriminant condition in general, and a singularity satisfying the discriminant condition is not  Teissier  in general. But the reason why we think that these singularities give a very good positive characteristics counterpart of quasi-ordinary singularities is the following result: assume for simplicity that $\K=\overline{\mathbf{F}}_p.$\\
  
\begin{Thm}\label{main} A Teissier singularity $(X,0)$ sits in an equisingular  family $\cX$ over $ \mathrm{Spec}(\cO_{\textbf{C}_p}) $ as a special fiber, and the generic fiber of $\cX$ has only quasi-ordinary singularities.
\end{Thm}

The field $\textbf{C}_p$ in the theorem is the completion of the algebraic closure of the field of $p-$adic numbers $\mathbf{Q}_p;$ it is of characteristic $0$, hence that the generic fiber is defined over a field of characteristic $0.$ Here, equisingular means that we have an embedded simultaneous resolution of $\cX.$ Although Teissier singularities are complicated in general, we can resolve their singularities thanks to the understanding of the neighboring quasi-ordinary singularities in characteristic zero that we can resolve by toric morphisms, inspired by the work of Gonz\'alez P\'erez \cite{Pedro}. Note also that any quasi-ordinary singularity in characteristic $0$ gives rise to a Teissier singularity.\\
Teissier singularities suggest to work on an analogue in characteristic $p$ of Jung's approach to resolution of singularities, where quasi-ordinary singularities are replaced by Teissier singularities. \\

The goal of this note is to define Teissier singularities and to give an idea of the proof 
of Theorem \ref{main} which is the main result of \cite{MS1}.

\smallskip

\section{The definition of Teissier singularities}

In this section, the field $\K$ is supposed to be algebraically closed, without any restriction on the characteristics. Let  
\[f(\x,z)=z^n+a_1(\x)z^{n-1}+\cdots + a_{n-1}(\x)z+a_{n}(\x) \in \K[[\x]][z],\]
be a Weierstrass polynomial, which will be assumed to be irreducible along the paper. With the projection $\pi$ defined by the algebraic morphism $\K[[\x]]\longrightarrow \K[[\x]][z]/(f),$ we will associate an invariant
$\kappa(\pi)$ which is a vector (it actually looks like a matrix) whose size depends on $f$ and whose components belong to $\IQ_{\geq 0}^d$, except the last one which is either $-1,$ or $\infty.$ Let us begin by explaining what is $\kappa(\pi).$ This follows \cite{MS} (which is defined in characteristic $0$)  but works well also in arbitrary characteristics and is explained in detail in \cite{MS1}.\\

We can rewrite $f$  as follows:

$$f(\x,z)=z^n+\sum_{(A,b)} c_{A,b}\x^Az^b,~ c_{A,b} \in \mathbf{K}.$$
We associate to $f$ and to the variables $\x,z$ a projected polyhedron as follows.

$$\Delta(f,\x,z):=\mbox{Convex hull}\left\{\frac{A}{n-b}+
\mathbf{R}_{\geq 0}^d\mid c_{A,b}\not=0\right\}$$
This convex body is very much related to the Newton polyhedron
which is the convex hull of the set given by $ (A,b) + \mathbf{R}_{\geq 0}^{d+1}.$ Indeed, the projected polyhedron is obtained by multiplying by $1/n$ (a dilation)  the projection of the Newton polyhedron from the point $(0,\ldots,0,n)$ (corresponding to the monomial $z^n$) to  $\mathbf{R}_{\geq 0}^d$ (corresponding to the first $d$ coordinates). Clearly, this object, the projected polyhedron, depends strongly on the coordinates $\x$ and $z$; in order to make it less dependent on the coordinate and capture information about $\pi,$ Hironaka considered the following object, called Hironaka’s characteristic polyhedron \cite{HiroPoly}:
$$\Delta(f,\x):=\bigcap_z \Delta(f,\x,z),$$
where the intersection ranges over all possible change of variables in $ z $.
	It was proved by Hironaka \cite{HiroPoly} that there exist a variable $z$ which satisfies
	$\Delta(f,\x)= \Delta(f,\x,z) $, see also \cite{CPcompl}.
	In \cite{MS}, we extend this notion by additionally varying the choice of $ \x $ appropriately, \textit{i.e.}, we introduce
$$\Delta := \Delta (f ) := \bigcap_{\x,z} \Delta(f,\x,z),$$
where $\x$ and $z$ are (change of variables) defined by automorphisms
of $\K[[\x]][z]$ that preserve $\K[[\x]].$  
We can generalize Hironaka's theorem and show that there exist variables $\x,z$ which satisfy

$$\Delta= \Delta(f,\x,z).$$
Another way to say it, there exists variables $\x,z$ which minimize the polyhedron $\Delta(f,\x,z)$ with respect to inclusion.

\begin{Ex}Let $f=(z-x)^2-x^3= z^2-2xz+x^2-x^3 \in \K[[\x]][z].$
A direct computation gives $\Delta(f,x,z)=[1,+\infty[ \subset \mathbf{R}_{\geq 0}.$ Now if we make the change of variables $y=z-x$ and keep $x$ we have the transform of $f$ that we abusively call $f$ too is $f=y^2-x^3.$ A direct computation gives $\Delta(f,x,y)=[3/2,+\infty[ \subset \mathbf{R}_{\geq 0}.$ So we have $\Delta(f,x,z) \subset \Delta(f,x,y)=[3/2,+\infty[.$  

\end{Ex}
From now on we fix the variables given by the generalization of Hironaka's theorem, those for which the projected polyhedron is minimal with respect to inclusion; the polyhedron is then called the \textbf{characteristic polyhedron}. We are ready to define the first component of the invariant $\kappa(\pi)$ which depends on the following three possibilities:
$$\left \{
    \begin{array}{ll}
      1)~ \mbox{If}~\Delta~ \mbox{has only one vertex}~ v_1~ \mbox{then the first component of}~\kappa(\pi)~\mbox{is}~ v_1. \\
      2) ~\mbox{If}~\Delta~ \mbox{is empty}~ \mbox{then}~\kappa(\pi)=(\infty). \\
    3) ~\mbox{If}~\Delta~ \mbox{has more than one vertex, then}~ \kappa(\pi)=(-1).
    \end{array}
    \right.
     $$
     
     In the second and third cases, we have defined $\kappa(\pi)$ completely. Assume that we are in the first case.
    We define the initial polynomial $In_{v_1}(f)$ of $f$ with respect to $v_1$ by the following:
    \begin{equation}
    In_{v_1}(f):=z^n+\sum_{\frac{A}{n-b}=v_1} c_{A,b}\x^Az^b.
\end{equation}  
Since $\K$ is algebraically closed,  $In_{v_1}(f)$ is a product of a monomial and binomials. Since $f$ is assumed to be irreducible, this implies that  $In_{v_1}(f)$ is a binomial raised at some power (see \cite{BGP,RS}), more precisely:
\begin{equation}
    In_{v_1}(f)=(z^{n_1}-c_{A_1}\x^{A_1})^{e_1}.
\end{equation} 
   Note that we have $n_1>1:$ indeed, if $n_1=1,$  by considering the variables $\x$ and $z-c_{A_1}\x^{A_1}$ (instead of $z$), we obtain that the projected polyhedron with respect to these new variables is strictly included in  the characteristic polyhedron with respect to the variables $\x$ and $z$; this contradicts the assumption that $\x,z$ verify Hironaka's theorem (and its extension in \cite{MS}). We conclude that
\begin{equation}
e_1<n,
\end{equation}  
since we have $n_1e_1=n.$ We introduce a new variable 
$$u_{1,0}:=z^{n_1}-c_{A_1}\x^{A_1}.$$
Let $f_{1,0}$ be the polynomial in $\K[[\x]][z][u_{1,0}]$ obtained from $f$ by replacing $z^{n_1}$ by $u_{1,0}+c_{A_1}\x^{A_1}.$ Hence $f_{1,0} \in \K[[\x]][z]_{<n_1}[u_{1,0}],$ \textit{i.e.}, the degree of $f_{1,0}$ in $z$ is strictly smaller than $n_1.$ Notice that the germ 
$$(X,0)=\{f=0\}\subset \mathrm{Spec} \K[[\x]][z]$$
is isomorphic to the germ defined by  $$\{f_{1,0}=u_{1,0}-(z^{n_1}-c_{A_1}\x^{A_1})=0\}\subset  \mathrm{Spec} \K[[\x]][z][u_{1,0}].$$ 
Write the expansion
$$f_{1,0}=u_{1,0}^{e_1}+\sum *_{A,b,c}\x^Az^bu_{1,0}^c.$$ 

We define 
$$\Delta(f_{1,0},\x,z,u_{1,0})=\mbox{Convex hull}\left\{\frac{A+bv_1}{e_1-c}+\mathbf{R}_{\geq 0}^d\mid *_{A,b,c}\not=0 \right\}.$$
We prove a Hironaka type theorem for this weighted projected polyhedron: 
\begin{Thm}\cite{MS,MS1}
There exists a system of variables $\x,z,u_{1}$ for which the polyhedron $\Delta(f_{1,0},\x,z,u_1)$ is minimal for inclusion and which is compatible with the characteristic polyhedron of $f,$ \textit{i.e}, for which the projected polyhedron of $f$ with respect to $\x,z$ is also the minimal polyhedron $\Delta$ chosen above.   
\end{Thm} Here is an idea of the proof:\\ we take the smallest vertex $v_{2,0}$ of $\Delta(f_{1,0},\x,z,u_{1,0})$ with respect to the lexicographical ordering and we try to eliminate it, \textit{i.e.}, to make a change of variables such that the projected polyhedron with respect to the new variables is smaller (with respect to inclusion) and does not have $v_{2,0}$ as vertex. The novelty concerns the change of variables involving $u_{1,0}.$ We consider the initial form of $f_{1,0}$  with respect to $v_{2,0}:$

\begin{equation}
    In_{v_2}(f_{1,0}):=u_{1,0}^{e_1}+\sum_{\frac{A+bv_{1}}{e_1-c}=v_{2,0}} *_{A,b,c}\x^Az^bu_{1,0}^c.
\end{equation} 

In general, this initial form is not a product of binomials, but, by Proposition 3.3 in \cite{MS}, we can canonically transform it into a  product of power of binomials modulo the equation $$u_{1,0}-(z^{n_1}-c_{A_1}\x^{A_1})=0.$$ We assume that it is a power of only one binomial, otherwise we put $v_2=-1.$ Write  

\begin{equation}\label{binomial}
 In_{v_{2,0}}(f_{1,0})=(u_{1,0}^{n_{2,0}}-*_{A_2,b_2,0}\x^{A_{2,0}}z^{b_{2,0}})^{e_{2,0}}.
 \end{equation}
 Here, there are no conditions which forbid $n_{2,0}$ from being equal to $1.$ If $n_{2,0}$ is not equal to $1$, we cannot eliminate the vertex $v_{2,0}$ by a change of variables involving $u_{1,0}$ and this $v_{2,0}$ will be simply called $v_2$ and we put $u_1=u_{1,0}.$ Assume, on the contrary, that $n_{2,0}=1;$ making the change of variables
  $$ u_{1,1}=u_{1,0}^{n_{2,0}}-*_{A_2,b_2,0}\x^{A_{2,0}}z^{b_{2,0}}=u_{1,0}-*_{A_2,b_2,0}\x^{A_{2,0}}z^{b_{2,0}},$$ 
  
 we eliminate $v_{2,0},$ \textit{i.e.}, $v_{2,0}$ is no longer a vertex of $\Delta(f_{1,1},\x,z,u_{1,1});$ here $f_{1,1}$ is obtained from $f_{1,0}$ by substituting      
$u_{1,0}$ by $u_{1,1}+*_{A_2,b_2,0}\x^{A_{2,0}}z^{b_{2,0}}.$ Consider $v_{2,1}$ to be the smallest vertex of $\Delta(f_{1,1},\x,z,u_{1,1})$ with respect to the lexicographical ordering; note that by construction we have a strict inclusion 
\begin{equation}\label{inclusion}
\Delta(f_{1,1},\x,z,u_{1,1})\subset \Delta(f_{1,0},\x,z,u_{1,0})
\end{equation} 
with respect to the product ordering (component by component). Again, we consider 
$In_{v_{2,1}}(f_{1,1})$ and we determine as in \eqref{binomial} $n_{2,1};$
if $n_{2,1} \not= 1$ then we cannot eliminate the vertex $v_{2,1}$ by a change of variables involving $u_{1,1}$ and this $v_{2,1}$ will be simply called $v_2$ and we put $u_1=u_{1,1}.$ If $n_{2,1}=1$ we make the change of variables  

$$u_{1,2}=u_{1,1}-*_{A_2,b_2,1}\x^{A_{2,1}}z^{b_{2,1}}=$$

$$u_{1,0}-*_{A_2,b_2,0}\x^{A_{2,0}}z^{b_{2,0}}-*_{A_2,b_2,1}\x^{A_{2,1}}z^{b_{2,1}}=$$

\begin{equation}
z^{n_1}-c_{A_1}\x^{A_1}-*_{A_2,b_2,0}\x^{A_{2,0}}z^{b_{2,0}}-*_{A_2,b_2,1}\x^{A_{2,1}}z^{b_{2,1}}
\end{equation}

where $A_{2,1}+b_{2,1}v_{1}=v_{2,1}.$
We continue this algorithm, and either at some time we will find some integer number $i$ for which $n_{2,i}>1$ or we construct a series 
\begin{equation}\label{converges}
u_{1,\infty}=z^{n_1}-c_{A_1}X^{A_1}-*_{A_2,b_2,0}\x^{A_{2,0}}z^{b_{2,0}}-*_{A_2,b_2,1}\x^{A_{2,1}}z^{b_{2,1}}-\cdots 
\end{equation}
Now recall that all the $b_j$ satisfy $b_j<n_1$ and that we have the strict inclusion \eqref{inclusion}: this tells us that the sum of the components of $A_j$ when $j$ tends to $\infty$, tends to $\infty$ and  that $u_{1,\infty}$ belongs to $\K[[\x]][z]_{<n_1},$ \textit{i.e.}, 
the coefficients of the powers of $z$ converges in $\K[[\x]]$ for the $\x$-adic topology generated by the powers of the maximal ideal. We put $u_1:=u_{1,\infty}$ and $f_1=f_{1,\infty}.$ Then the smallest vertex $v_{2}$ of $\Delta(f_{1},\x,z,u_{1})$ with respect to the lexicographical ordering cannot be eliminated by a change of variables involving $u_1.$ Continuing the same game with the other vertices of $\Delta(f_{1},\x,z,u_{1})$ and making use of the allowed changes of variables which involve the $\x$ only without disturbing the first characteristic polyhedron $\Delta,$ we end up by obtaining a weighted characteristic polyhedron  $\Delta_1.$   

We are ready to define the second component of the invariant $\kappa(\pi)$ which depends on the following three possibilities:
$$\left \{
    \begin{array}{ll}
      1)~ \mbox{If}~\Delta_1~ \mbox{has only one vertex}~ v_2~ \mbox{then the second component of}~\kappa(\pi)~\mbox{is}~ v_2. \\
      2) ~\mbox{If}~\Delta_1~ \mbox{is empty}~ \mbox{then}~\kappa(\pi)=(v_1,\infty). \\
    3) ~\mbox{If}~\Delta_1~ \mbox{has more than one vertex, then}~ \kappa(\pi)=(v_1,-1).
    \end{array}
    \right.
     $$
If we are in the first situation, \textit{i.e.}, $\Delta_1$ has only one vertex $v_2,$ the choice of $u_1$ insures that in the equation
 \begin{equation}
 In_{v_{2}}(f_{1})=(u_{1}^{n_{2}}-c_{A_2,b_2^{(0)}}\x^{A_2}z^{b_2^{(0)}})^{e_{2}},
 \end{equation}
 $n_2>1$ and hence $e_2<e_1<e_0:=n.$ Similarly, we continue the construction as above, we introduce a new variable  $$u_{2,0}:=u_{1}^{n_{2}}-c_{A_2,b_2^{(0)}}\x^{A_2}z^{b_2^{(0)}}$$
 and the polynomial $f_{2,0}\in\K[[\x]][z]_{<n_1}[u_1]_{<n_2}[u_{2,0}],$ obtained from $f_1$ by substituting $u_1^{n_2}$ by $u_{2,0}+c_{A_2,b_2^{(0)}}\x^{A_2}z^{b_2^{(0)}}.$ The procedure, stops in one of two cases: either we get to a situation where our weighted polyhedron has more than one vertex; or we construct a strictly increasing sequence 
 $$n>e_1>e_2>\cdots,$$
 which should end by $e_{g}=1,$ for some $g\in \mathbb{N}$ (it is the case where the weighted polyhedron is empty).
We are now ready to define Teissier singularities.

\begin{Def}
Let  $\textbf{K}$ be an algebraically closed field. Let $f\in \textbf{K}[[\x]][z]$ be a Weierstrass polynomial satisfying $f(0)=0.$ The hypersurface singularity $(X,0)$ defined by  $f=0$, equipped with the map $(X,0)\longrightarrow (\mathbb{A}^d,0)$ given by $(\x,z)\mapsto \x$, is a \textbf{Teissier singularity} if the last component of $\kappa(\pi)$ is $\infty.$
\end{Def}
It follows from \cite{MS} that in characteristic $0,$ $(X,0)\longrightarrow (\mathbb{A}^d,0)$ is quasi-ordinary if and only if it is Teissier. So it is in positive characteristic that these singularities are essentially new. Indeed let us consider the following examples in characteristic $2.$

\begin{Ex}\label{first example}
\begin{enumerate}
\item Let $\K=\overline{\textbf{F}}_2$ and $f=z^2-x_1x_2z-x_1^3x_2-x_1x_2^3\in \K[[x_1,x_2]][z].$ A 
direct computation allows the reader to see that the projected 
polyhedron $\Delta(f,\x,z)$ is minimal and that it has two vertices. So the singularity $(X,0)=\{f=0\}\longrightarrow (\textbf{A}^2,0)$ 
is not Teissier since we have $\kappa(\pi)=(-1).$ At the same time, a direct 
calculation of the discriminant $\mathcal{D}_z(f)$ of $f$ with respect to $z$ gives $\mathcal{D}_z(f)=(x_1x_2)^2.$ So $(X,0)$ is
quasi-ordinary but not Teissier.
\item Let $\K=\overline{\textbf{F}}_2$ and $f=z^2-x^3\in \K[[x_1,x_2]][z].$
We have $\kappa(\pi)=(\frac{3}{2},\infty);$ so  $(X,0)=\{f=0\}\longrightarrow (\textbf{A}^1,0)$ is Teissier. But $(X,0)=\{f=0\}\longrightarrow (\textbf{A}^1,0)$ is not quasi-ordinary since $\mathcal{D}_z(f)=0$ and the extension defined by $f$ is purely inseparable.   
\end{enumerate}
\end{Ex}
\section{Simultaneous embedded resolution} 
From now on, we assume that the characteristic $p$ of $\K$ is positive.
As we just saw, in positive characteristics, Teissier singularities and quasi-ordinary singularities can be very different. But as we will show below, they are comparable from a deformation 
theoretic viewpoint. First, notice that it follows from the definition of a Teissier singularity $(X,0)\longrightarrow (\mathbf{A}^d,0)$ that $(X,0)$ is isomorphic to a subvariety of $$\mathbf{A}^{d+g}=\mathrm{Spec}\K[[\x]][z,u_1,\ldots,u_{g-1}]$$ defined by an ideal $I\subset \K[[\x]][z,u_1,\ldots,u_{g-1}]$ which is generated by the following $g$ functions:

\begin{equation}
	\label{equations}
\left\{ 
\begin{array}{l}
	u_1-(z^{n_1}-c_{A_1}\x^{A_1}+h_1(\x,z,u_1)),  \\
	u_2-(u_1^{n_2} -c_{A_2,b_2^{(0)}}\x^{A_2}z^{b_2^{(0)}}+h_2(\x,z,u_1,u_2)),\\
	\vdots \\
	u_{g-1}-(u_{g-2}^{n_{g-1}}-c_{A_{g-1},b_{g-1}^{(0)},\bf{b_{g-1}}}\x^{A_{g-1}}z^{b_{g-1}^{(0)}}\u^{\bf{b_{g-1}}}+h_{g-1}(\x,z,u_1,\ldots,u_{g-2}))
	\\
	u_{g-1}^{n_g}-c_{A_{g},b_{g}^{(0)},\bf{b_{g}}}\x^{A_{g}}z^{b_{g}^{(0)}}\u^{\bf{b_{g}}}+h_{g}(\x,z,u_1,\ldots,u_{g-1}), 
\end{array} 
\right.
\end{equation}

where for $i=3,\ldots,g ,\bf{b_{i}}$ is a vector having $i$ components and $\u^{\bf{b_{i}}}$ is the monomial 
$$u_1^{b_{i}^{(1)}}\cdots u_{i}^{b_{i}^{(i)}}$$ and where if we give the variable $x_i$ the weight $e_i:=(0,\ldots,0,i,0,\ldots,0)$ (the    $i$'th vector of a standard basis of $\mathbf{R}^d),$ for $i=1,\ldots,d,$ the variable $z$ the weight $v_1$ and the variable $u_i$ the weight $v_{i+1}$ for $i=1,\ldots,g-1,$ the weights of the monomials appearing in the power series $h_i,i=1,\ldots,g$ are larger with respect to the product ordering than
 $$n_iv_{i}=\mathrm{weight}(u_{i-1}^{n_i})=\mathrm{weight}(
\x^{A_i} z^{b_{i}^{(0)}}\u^{\bf{b_{i}}}).$$
 This is an overweight deformation of the ideal generated by the initial binomials.


Since $ \K $ is algebraically closed, it is perfect in particular.
Therefore, by \cite{LocalFields}, [II, \S 5, Theorem 3],
there exists a complete discrete valuation ring $ \cO $ of mixed characteristics
which is absolutely ramified (i.e., $ p $ is a local uniformizer of $ \cO$)
and has $ \K $ as its residue field. 
Set 
\[
	\cT := \Spec (\cO).
\]
Let $ 0 \in \cT $ be the closed point and $ \eta \in \cT $ be the generic point; note that $\cT$ is of dimension $1$ and should thought as a curve.

\begin{Prop}
	\label{Prop:family_for_characteristic_change}
	There exist a scheme $ \cX $ and a morphism $ \psi \colon \cX \to \cT $ such that  
	$ \pi $ is flat,	
	$ \cX_0 := \psi^{-1} (0) \cong X $ and $ \cX_\eta := \psi^{-1} (\eta) $ is isomorphic to a quasi-ordinary hypersurface singularity. 
\end{Prop} 

\begin{proof}
	We lift $ I \subset k[[\x]][z,u_{1},\ldots,u_{g-1}] $ to
	\[
		\cI := I \cdot \cO[[\x]][z,\u_{1},\ldots,u_{g-1}].
	\]
	Let $ \cX $ be the scheme defined by $ \cI $,
	\[
		\cX := \Spec (R),
		\ \ \ 
		\mbox{ where } 
		R := \cO[[\x]][z,\u_{1},\ldots,u_{g-1}]/ \cI,
	\] 
	and let $ \psi \colon \cX \to \cT $ be the natural projection.
	Since $ p $ is not a zero divisor in $ R $,
	the morphism $ \psi $ is flat. 
	%
	%
	%
	By construction, we have $ \cX_0=\psi^{-1}(0) \cong X $ 
	and moreover, by \cite{MS}, the generic fiber $ \cX_\eta:=\psi^{-1}(\eta) $ is isomorphic to a quasi-ordinary hypersurface singularity. 
\end{proof}

In other terms, $\cI$ is the ideal generated by the elements given in \eqref{equations}; it is a very particular lift of the defining ideal of $I.$ The main goal of this article is to announce the following result. 

\begin{Thm}\label{main1}The family $\cX\longrightarrow \cT$  admits a simultaneous embedded resolution of singularities.
\end{Thm}
More precisely, let $Z=\mathrm{Spec}\cO[[\x]][z,u_1,\ldots,u_{g-1}]$ and let $\overline{\psi}:Z\to \cT$ be the natural projection that extends $\psi;$ let $Z_t:={\overline{\psi}}^{-1}(t).$ There exists a birational proper morphism 
$\mu:\widetilde{Z}\longrightarrow Z,$ such that  $\mu_{\mid\mu^{-1}(Z_t)}$ for $t\in \cT$ is a birational toric morphism defined by a fan $\Sigma$ which is a subdivision of $\mathbf{R}_{\geq 0}^{d+g}$ and which is independent of $t\in \cT.$ 
This map $\mu$ satisfies that  $\mu_{\mid\mu^{-1}(Z_t)}$ is an embedded resolution of $\cX_t;$ moreover, the exceptional divisor $E$ of $\mu$ is a relative normal crossing divisor over $\cT;$ in particular $E$ is a flat family over $\cT.$ This means that a Teissier singularity is equisingular in the sense of the theorem to a carefully chosen quasi-ordinary singularity 
whose defining equation contains monomials that one cannot see in characteristic $p$ (over $t=0$) but which are very important to resolve the Teissier singularities. We call these monomials the \textbf{ghost monomials}, since they disappear in characteristic $p$ (over $t=0$) but appear in characteristic $0$ (over $t=\eta$).

\begin{Ex}\label{purinsep}
Let $p=2$ and let $\K=\overline{\bf{F}}_2.$ Then $\mathcal{O}$ is the valuation ring of the $p-$adic complex number $\mathbf{C}_p.$
Let $f=(z^2-x_1^3x_2^6)^4-x_1^{15}x_2^{30} \in \K[[x_1,x_2]][z].$
The singularity 
$$(X,0)=\{f=0\}\to \mathrm{Spec}\K[[x_1,x_2]]$$ is Teissier and we have $$\kappa(f)=((3/2,3),(15/4,15/2),(63/8,63/4)).$$ Indeed, after computing the first characteristic polyhedron, we introduce the new variable 
\begin{equation}\label{ou1}
u_1=z^2-x_1^3x_2^6.
\end{equation}
we find then, after substituting $z^2$ by $u_1+x_1^3x_2^6$ that
$$f_1=u_1^4-x_1^{15}x_2^{30}$$
$$= u_1^4-x_1^{12}x_2^{24}x_1^3x_2^6$$
which is equal modulo the equation (\ref{ou1}) to
$$ u_1^4-x_1^{12}x_2^{24}(z^2+u_1)=(u_1^2-x_1^6x_2^{12}z)^2-x_1^{12}x_2^{24}u_1.$$

We introduce the variable $$u_2=u_1^2-x_1^6x_2^{12}z.$$
This gives us that our singularity $(X,0)$ is isomorphic to the singularity defined in $\mathrm{Spec}\K[[x_1,x_2]][z,u_1,u_2]$ by
the ideal $I$ generated by:
$$
\begin{array}{ll}
    u_1-(z^2-x_1^3x_2^6),\\
     u_2-(u_1^2-x_1^6x_2^{12}z),\\
u_2^2+x_1^{12}x_2^{24}u_1.
     \end{array} 
$$

This also allows us to rewrite the equation $f=0$ as follows:
$$((z^2-x_1^3x_2^6)^2-x_1^6x_2^{12}z)^2+x_1^{12}x_2^{24}(z^2-x_1^3x_2^6)=0.$$
Notice that the coefficient of the monomial $x_1^{12}x_2^{24}z^2$ is zero in characteristic $2$ but non-zero in characteristic $0.$ The monomial $x_1^{12}x_2^{24}z^2$ is a ghost monomial.
\end{Ex}   

There are two important steps in the proof of Theorem \ref{main1}.\\

In first step, one needs to prove that the singular locus of $\cX_t\subset Z_t$ is included in the union of the coordinate hyperplanes of $Z_t,$ \textit{i.e.,} the hyperplanes defined by
the vanishing of one of the coordinates $x_1,\ldots,x_d,z,u_1,\ldots,u_{g-1}$. For $g=1,$ this is proved by a direct application of the Jacobian criterion. For $g>1,$ this is proved by induction on $g$ and for that we had to introduce Teissier singularities defined on toric varieties; this generalizes in some sense \cite{Pedro2} to positive characteristics.\\

The second steps is to prove, along the ideas of \cite{Tevelev,Aroca1,AGM} that $\cX_t\subset Z_t$ is a Newton 
non-degenerate singularity for any $t \in \cT$ and that moreover for any vector $\omega \in \mathbf{R}_{\geq 0}^{d+g},$ the initial ideal of $\cI$ is equisingular over $\cT.$ For instance, in the example \eqref{purinsep}, if we take a vector $\omega$ whose components are strictly positive real numbers and which lies in the local tropical variety of $(\cX_t,0)\subset \mathbf{A}^{d+g}$ (these are the important types of vectors for studying non-degeneracy: indeed for the other vectors, the initial ideals contain monomials, so the varieties they define, in particular their singular loci, are included in coordinate hyperplanes), then the initial ideal of $\cI_t$  with respect to the weight determined by $\omega$ is generated by:
$$
\begin{array}{ll}
    z^2-x_1^3x_2^6,\\
     u_1^2-x^6x_2^{12}z,\\
u_2^2+x_1^{12}x_2^{24}u_1.
     \end{array} 
$$
The variety defined by this initial ideal is toric and its singular locus is included in the union of the coordinate hyperplanes; moreover one notices that this ideal is "independent" of $t,$ even though depending on $t$, the initial ideal is included in a polynomial ring with the same variable but over different fields (recall here that the closed point of $\cT$ is defined by the ideal generated by $(p)).$ The fact the the initial ideal with respect to such a vector $\omega$ is "independent" of $t$ will imply that we can construct a simultaneous "toric" embedded resolution of $\cX\subset \mathbf{A}_\cO^{d+g}.$

\bibliographystyle{alpha}
\bibliography{TeissierBib}

\end{document}